\documentclass[11pt,a4paper,reqno]{amsart}
\usepackage{epsfig}
\usepackage{color}
\usepackage{amssymb,amsmath,amsthm,amstext,amsfonts}
\usepackage[normalem]{ulem}
\usepackage{amsmath,amscd,amstext,amsthm,amsfonts,latexsym}
\usepackage[colorlinks=true,linkcolor=blue, urlcolor=red, citecolor=blue]{hyperref}	
\usepackage{amsmath,amssymb,amsthm}
\usepackage{color}
\usepackage[active]{srcltx}
\usepackage{a4wide}
\usepackage{amssymb, amsmath}
\usepackage{graphicx}
\usepackage{float}
\usepackage[active]{srcltx}

\def \N {\mathbb{N}}
\def \vep {\varepsilon}

\theoremstyle{plain}

\newtheorem{theorem}{Theorem}[section]
\newtheorem{lemma}[theorem]{Lemma}
\newtheorem{proposition}[theorem]{Proposition}

\newtheorem{maintheorem}{Theorem}
\newtheorem{maincorollary}{Corollary}

\newtheorem{remark}[theorem]{Remark}

\numberwithin{equation}{section}

\begin{document}

\title[Historic behavior and discontinuous first integrals]{Genericity of historic behavior \\ for maps and flows}

\author[M. Carvalho]{Maria Carvalho}
\address{CMUP \& Departamento de Matem\'atica, Faculdade de Ci\^encias da Universidade do Porto \\
Rua do Campo Alegre 687, 4169-007 Porto, Portugal.}
\email{mpcarval@fc.up.pt}

\author[P. Varandas]{Paulo Varandas}
\address{ CMUP, Universidade do Porto \& Departamento de Matem\'atica, Universidade Federal da Bahia\\
Av. Ademar de Barros s/n, 40170-110 Salvador, Brazil}
\email{paulo.varandas@ufba.br}




\maketitle

\begin{abstract}
We establish a sufficient condition for a continuous map, acting on a compact metric space, to have a Baire residual set of points exhibiting historic behavior (also known as irregular points). This criterion applies, for instance, to a minimal and non-uniquely ergodic map; to maps preserving two distinct probability measures with full support; to non-trivial homoclinic classes; to some non-uniformly expanding maps; and to partially hyperbolic diffeomorphisms with two periodic points whose stable manifolds are dense, including Ma\~n\'e and Shub examples of robustly transitive diffeomorphisms. This way, our unifying approach recovers a collection of known deep theorems on the genericity of the irregular set, for both additive and sub-additive potentials, and also provides a number of new applications.
\end{abstract}

\section{Introduction}

Let $f:\,(X, \mathcal{A}) \to (X, \mathcal{A})$ be a measurable transformation and $\mu$ be an $f$-invariant ergodic probability measure on the $\sigma$-algebra $\mathcal{A}$. For a measurable function $\varphi:\,X \to \mathbb{R}$ and $x \in X$, the sequence of Birkhoff averages of $\varphi$ at $x$ is given by $\big(1/n\,\sum_{j=0}^{n-1}\, \varphi(f^j(x))\big)_{n \in \mathbb{N}}$. A point $x \in X$ is said to be $\varphi$-regular if the limit of this sequence exists; otherwise $x$ is called a $\varphi$-irregular point (and said to have historic behavior \cite{Ruelle, Takens}). Birkhoff's ergodic theorem asserts that the time averages of $\varphi$ at $\mu$-almost every point in $X$ converge to the space average $\int_X \varphi \, d\mu$. So the set of $\varphi$-regular points carries full $\mu$ measure. This result supports Boltzman ergodic hypothesis but fails to describe the behavior and the complexity of the set of points at which the sequence of Birkhoff averages has no limit. Nowadays there is a well established theory to assess how big is the \emph{irregular set} (also called the set of points with historic behavior): contrary to the previous measure theoretical description, the set of these non-typical points may be Baire generic and, moreover, have full topological pressure, full Hausdorff dimension or full metric mean dimension, as attested in
\cite{AP19,BLV, BS, LV18, PS07, TV17} and references therein.

\smallskip

Not surprisingly, the research addressing these phenomena started in the realm of uniform hyperbolicity. On the one hand, the existence of many periodic points ensures that there exist continuous observable maps with distinct spaces averages. On the other hand, a hyperbolic structure, even if non-uniform, comes up with a panoply of means to reconstruct true orbits from 
finite pieces of orbits: specification, gluing orbit property and non-uniform versions of these. Yet, these properties are seldom valid for strong partially hyperbolic transitive diffeomorphisms (see \cite{BTV21,SVY} and references therein). 
Furthermore, according to \cite{Sun}, no minimal and positive entropy homeomorphism has the gluing orbit property, which is one of the weakest versions on the aforementioned chain of concepts. This may explain the scarcity of results, as far as we know, regarding the irregular set for minimal dynamics.

\smallskip

One may think of minimal and uniquely ergodic dynamics as the natural opposites to hyperbolic dynamics, with much lower level of complexity. This impression is somehow reinforced by the fact that minimal dynamics admit very simple Rohklin towers (see e.g. \cite[Lemma 6]{AB}), and by the uniform convergence of Birkhoff averages in the case of uniquely ergodic transformations. Notwithstanding, examples of volume preserving analytic diffeomorphisms on $\mathbb T^2$ with zero entropy, which are minimal though not uniquely ergodic (due to Furstenberg \cite{Furstenberg}), or minimal homeomorphisms of the torus with positive entropy (constructed by Herman and Rees \cite{H81, R81}) show that the setting is surprisingly rich.

\smallskip

In this paper we present a simple criterion to ensure that the set of points with historic behavior, for maps or flows, is a Baire generic subset of the ambient space, both for additive and sub-additive sequences. The main results are presented in Section~\ref{se:main-results}. Their proofs are a consequence of the existence of an everywhere discontinuous first integral and of a general statement, we will show on Section~\ref{se:teo-p}, regarding the accumulation points of sequences of continuous observable functions.
This reasoning provides a unified approach to several contexts where this kind of result has already been established, besides bringing forward new applications. Indeed, our assumptions are satisfied by a vast class of discrete and continuous time dynamics, including minimal non-uniquely ergodic homeomorphisms, non-trivial homoclinic classes, continuous maps with the specification property, Viana maps, partially hyperbolic diffeomorphisms and singular-hyperbolic attractors (cf. Corollaries \ref{main-b} -- \ref{main-g}). In the broader context of sub-additive sequences of potentials, we show how to describe more accurately the irregular points for Lyapunov exponents of linear cocycles, as well as those points for which the convergence in the Brin-Katok formula for the metric entropy of weak Gibbs measures fails. Therefore, we establish or improve a wide range of results on the genericity of the irregular set, including \cite{BLV, BKNRS, Y2020}, as an outcome of an easy criterion.


\medskip

\section{Main results}\label{se:main-results}

Given a compact metric space $(X,d)$ and a continuous map $f: X \to X$, the set of irregular points is defined by
\begin{equation}\label{eq:IG}
\mathfrak I = \Big\{x \in X  \colon \,\lim_{n \,\to \,+\infty} \,\frac1n\sum_{j=0}^{n-1} \,\delta_{f^j(x)} \quad \mbox{does not exist in the weak$^*$ topology}\, \Big\}.
\end{equation}
For each continuous function $\varphi : X  \to \mathbb R$, the set of $\varphi$-irregular points is given by
\begin{equation}\label{eq:IF}
\mathfrak I_\varphi = \Big\{x \in X  \colon\, \Big(\frac1n\,\sum_{j=0}^{n-1} \,\varphi({f^j(x)})\Big)_{n \, \in \, \mathbb{N}} \quad \mbox{does not converge}\, \Big\}.
\end{equation}
Notice that, as $\varphi$ is bounded, the sequence of Birkhoff averages of $\varphi$ at $x$ does not converge if and only if it has no limit. Moreover, $\mathfrak I_\varphi \subset \mathfrak I$. Associated to $\varphi$, consider the map
\begin{equation}\label{eq:function L}
x \in X \quad \mapsto \quad L_\varphi(x)=\limsup_{n\,\to\,+ \infty} \,\frac1n \,\sum_{j=0}^{n-1} \,\varphi(f^j (x)).
\end{equation}
Observe that $L_\varphi$ is $f$-invariant, that is, $L_\varphi \circ f = f$. Thus, $L_\varphi$ is a so called first integral for $f$; for more information regarding smooth first integrals for endomorphisms, see \cite{Hurley86, Mane73}.

\smallskip

As a particular case of a more general statement we prove on Section~\ref{se:teo-p}, we will show that the existence of a non-trivial and discontinuous first integral $L_\varphi$ conveys a topologically large set of irregular points. To the best of our knowledge, this sufficient condition has not appeared before in the literature, so it primarily provides a new criterion for genericity of irregular sets.

\begin{maintheorem}\label{thm:main}
Let $(X,d)$ be a compact metric space, $f: X \to X$ be a continuous map and $\varphi: X \to \mathbb{R}$ be a continuous observable. Assume that there exist two
dense subsets $\mathcal{A}$ and $\mathcal{B}$ of $X$ such that the restrictions of $L_\varphi$ to $\mathcal{A}$ and to $\mathcal{B}$ are constant, though the value at $\mathcal{A}$ is different from the one at $\mathcal{B}$. Then $\mathfrak I_\varphi$ is a Baire residual subset of $X$.
\end{maintheorem}


\smallskip

Let us now list a few applications of Theorem~\ref{thm:main} in a variety of settings. Given a measurable map $f: X \to X$ and an $f$-invariant probability measure $\mu$, the ergodic basin of attraction of $\mu$ is the set of points $x \in X$ such that $\lim_{n \,\to \,+\infty} \,\frac1n\,\sum_{j=0}^{n-1} \,\delta_{f^j(x)} = \mu$ in the weak$^*$ topology. A first consequence of Theorem~\ref{thm:main} is the following result, whose main assumption is fulfilled, for instance, whenever $f$ preserves two distinct probability measures with full support.

\begin{maincorollary}\label{main-b}
Let $(X,d)$ be a compact metric space and $f: X \to X$ be a continuous map preserving two distinct Borel probability measures with dense ergodic basins. Then $\mathfrak I$
is a Baire residual subset of $X$.
\end{maincorollary}

\smallskip

Corollary~\ref{main-b} is suited, for example, to continuous maps with the specification property or the gluing orbit property (we refer the reader to \cite{BV15} for the definitions), providing a simpler and alternative proof of the genericity of the irregular set in that setting. Indeed, if the continuous map satisfies the specification property then \cite[Propositions 21.12 and 21.14]{DGS} ensure not only that generic invariant measures are full supported but also that their basins are dense.

\begin{maincorollary}\label{cor-spec}
Let $X$ be a compact metric space and  $f : X \to X$ be a continuous map satisfying the specification property. There exists a Baire generic subset of points in $X$ with historic behavior.
\end{maincorollary}

\smallskip

Corollary~\ref{main-b} also applies to endomorphisms with critical or singular behavior (e.g. quadratic maps, Lorenz interval maps or Viana maps) under just a few requirements. Theorem~5 in \cite[p. 928]{Pi} indicates that any strongly transitive $C^{1+\alpha}$-map on a compact Riemannian manifold with a periodic point which does not lie in the forward orbit of the critical or singular set admits an uncountable number of ergodic and full supported invariant measures. Let us illustrate this assertion by considering the robust class of multidimensional non-uniformly expanding maps with singularities known as Viana maps. These are skew-products of the type
\begin{equation}\label{eq:Viana}
\begin{array}{cccc}
f : & \mathbb S^1 \times \mathbb R & \to & \mathbb S^1 \times \mathbb R \\
	& (x,y) & \mapsto & \big(dx \,(mod \,1),\,a(x)-y^2\big)
\end{array}
\end{equation}
where $d \geqslant 2$, $a(x)=a_0+\alpha \sin(2\pi x)$ and $a_0\in (1,2)$ is chosen so that $0$ is pre-periodic for the quadratic map $h(x)=a_0-x^2$ (cf. \cite{Viana} for more details). According to  \cite[Theorem~10]{Pi}, Viana maps have an uncountable number of full supported ergodic probability measures, and so we obtain the following immediate consequence.

\begin{maincorollary}\label{main-viana}
Consider a map $f : \mathbb S^1 \times \mathbb R \to \mathbb S^1 \times \mathbb R$ as in ~\eqref{eq:Viana}. There exists a Baire generic subset of points in the attractor $\Lambda = f(\mathbb S^1 \times \mathbb R)$ with historic behavior.
\end{maincorollary}

\smallskip

It is known that the irregular set associated to a uniquely ergodic dynamics is empty (cf. \cite[Theorem 6.19]{Walters}). Besides, some of the most interesting known examples of minimal non-uniquely ergodic homeomorphisms have zero topological entropy. Thus, in this setting, it is useful to describe the complexity of the irregular set using Baire category arguments instead of other measurements of chaos. The next consequence of Corollary~\ref{main-b} benefits precisely from this strategy.

\begin{maincorollary}\label{main-a}
Let $(X,d)$ be a compact metric space and $f: X \to X$ be a continuous minimal map. Then either there exists an $f$-invariant ergodic Borel probability measure $\mu$ such that, in the weak$^*$ topology,
$$\lim_{n \,\to \,+\infty} \frac1n\,\sum_{j=0}^{n-1} \delta_{f^j(x)}  = \mu	\quad \quad \forall \, x \in X$$
or the set $\mathfrak I$ is Baire residual in $X$.
\end{maincorollary}

\smallskip

Corollary~\ref{main-a} applies, for instance, to the minimal and non-uniquely ergodic homeomorphisms constructed by Furstenberg \cite{Furstenberg}. Its statement appears, though with a distinct formulation, in \cite[Proposition 6.3]{H81} and \cite[Lemma 3]{Fu97}.

\smallskip

A third consequence of Theorem~\ref{thm:main} concerns the irregular set for partially hyperbolic diffeomorphisms. Given a smooth compact manifold $X$ with dimension $\dim (X) \geqslant 2$, one says that a diffeomorphism $f\in \mathrm{Diff}^{\, 1}(X)$ is \emph{partially hyperbolic} if there exists a $Df$-invariant splitting $TX=E^s\oplus F$ and constants $C>0$ and $\lambda\in (0,1)$ such that, for every $x \in X$ and every $n \in \mathbb{N}$, one has:
\begin{enumerate}
\item[(a)] $\|Df^n(x)(v)\|\leqslant C\lambda^n\|v\|$ for every $v\in E^s_x$;
\smallskip
\item[(b)] for every pair of unitary vectors $v \in E^s_x$ and $w \in F_x$,
$$\frac{\|Df^n(x) v\|}{\|Df^n(x) w\|} \leqslant C \lambda^n.$$
\end{enumerate}
Property (a) means that the sub-bundle $E^s$ is uniformly contracting, while a splitting satisfying property (b) is called \emph{dominated}. It is well known that, under these assumptions, for every point $x\in X$ there exists a $C^1$-submanifold called stable manifold $W^s(x)$ passing at $x$ and tangent to $E^s_x$, and the collection of stable manifolds defines a stable foliation on $X$.
In several relevant examples of partially hyperbolic diffeomorphisms it is known that the stable foliation is minimal, meaning that all the stable manifolds are dense in $X$. We will establish the following variant of Theorem~\ref{thm:main} within the setting of those partially hyperbolic diffeomorphisms.

\begin{maincorollary}\label{main-c}
Let $X$ be a compact Riemannian manifold and $f: X \to X$ be a partially hyperbolic diffeomorphism with two distinct periodic points whose stable manifolds are dense in $X$. Then $\mathfrak I$
is Baire residual in $X$.
\end{maincorollary}

\smallskip

This corollary applies, for instance, to the open set of robustly transitive diffeomorphisms on $\mathbb{T}^4$ with minimal stable foliations constructed by Shub in \cite{Shub1971} and to Ma\~n\'e's examples \cite{Mane}. A dual statement for partially hyperbolic diffeomorphisms having a dominated splitting into center and unstable sub-bundles with a minimal unstable foliation can be proved similarly.

\smallskip

It is worth noticing that our results also allow us to deal with important classes of proper subsets that are invariant by the dynamics, as is the case of non-trivial homoclinic classes. If $X$ is a compact Riemannian manifold, the \emph{homoclinic class} $H(p,f)$ associated to a hyperbolic saddle periodic point $p$ of a diffeomorphism $f \in \text{Diff}^{\,1}(X)$ is defined as the closure of transverse intersections of the unstable and the stable manifolds of the orbit of $p$. It is known that there exists a Baire residual subset of irregular points on each non-hyperbolic homoclinic class
of $C^1$-generic diffeomorphisms of $X$ (see \cite{BKNRS}). The next modification of Corollary~\ref{main-c} improves \cite{BKNRS} since it actually applies to arbitrary non-trivial homoclinic classes of any $C^1$-diffeomorphism.

\begin{maincorollary}\label{main-homoclinic}
Let $X$ be a compact Riemannian manifold. Given  $f\in \text{Diff}^{\,1}(X)$ and a hyperbolic saddle periodic point $p$ by $f$, then either $H(p,f)=\{p\}$ or the set $\mathfrak I \cap H(p,f)$ is a Baire residual subset of $H(p,f)$.
\end{maincorollary}

\smallskip

The previous result complements \cite[Proposition 9.1]{ABC2011}, which asserts that the set of points with historic behavior is Baire generic in the closure $\overline{W^s(\mathcal O(p))}$ of the basin of attraction
of any non-trivial homoclinic class $H(p)$. This proposition does not impart the same information of the previous corollary since it is not directly applicable to the homoclinic class itself and $H(p)$ is a meager subset of $\overline{W^s(\mathcal O(p))}$ even in the hyperbolic context. However, the proof of \cite[Proposition 9.1]{ABC2011} may be adjusted to convey the statement of Corollary~\ref{main-homoclinic}.

\smallskip

At this moment, it is natural to ask whether Theorem~\ref{thm:main} and its corollaries can be adapted to deal with the broader context of non-additive sequences, which are relevant to the computation of several dynamical quantities, such as entropy and Lyapunov exponents. The case of almost additive or asymptotically additive sequences of continuous functions carries no extra difficulties (cf. Remark~\ref{rmk51}); on the contrary, the case of continuous sub-additive sequences requires further explanation.

\smallskip

A sequence $\Phi=(\varphi_n)_{n\in \mathbb N}$ of continuous maps $\varphi_n: X \to \mathbb R$ is \emph{sub-additive} if
$$\varphi_{m+n} \leqslant \varphi_m \circ f^n+ \varphi_n \quad \quad \forall \,m,n \,\in\, \mathbb N.$$
Accordingly, the set of $\Phi$-irregular points is defined by
\begin{equation}\label{eq:IFP}
\mathfrak I_\Phi = \Big\{x \in X  \colon\, \Big(\frac1n\, \varphi_n(x) \Big)_{n \, \in \, \mathbb{N}} \quad \mbox{does not converge}\, \Big\}.
\end{equation}
Kingman's sub-additive ergodic theorem guarantees that $\mathfrak I_\Phi$ has zero measure with respect to every $f$-invariant Borel probability measure on $X$. In this context, the map $M_\Phi$ defined by
\begin{equation}\label{eq:function LP}
x \in X \quad \mapsto \quad M_\Phi(x) = \inf_{n\,\in\, \mathbb N} \,\frac1n \,\varphi_n (x)
\end{equation}
is a Lyapunov function associated to the discrete dynamical system $f$. More precisely, the inequality
$$\frac1{n+1} \,\varphi_{n+1}(x) \,\leqslant\, \frac1{n+1} \,\big[\,\varphi_n (f(x)) + \sup_{z\,\in\, X} \,\varphi_1(z)\, \big]$$
ensures that $M_\Phi(x) \leqslant M_\Phi(f(x))$ at every $x \in X$.

\smallskip

Although the function $M_\Phi$ is measurable, because it is the infimum of a sequence of continuous functions, in general one can not rely on higher regularity. This is an evidence that the previous concept of Lyapunov function scarcely describes the irregular set $\mathfrak I_\Phi$, though it suggests the following version of Theorem~\ref{thm:main} for sub-additive sequences.

\begin{maintheorem}\label{thm:mainB}
Let $(X,d)$ be a compact metric space, $f: X \to X$ be a continuous map and $\Phi=(\varphi_n)_{n\in \mathbb N}$ be a sub-additive sequence of continuous maps on $X$. Assume that
there exist two dense subsets $\mathcal{A}$ and $\mathcal{B}$ of $X$ such that the restrictions of $M_\Phi$ to $\mathcal{A}$ and to $\mathcal{B}$ are constant, though the value at $\mathcal{A}$ is different from the one at $\mathcal{B}$. Then $\mathfrak I_\Phi$ is Baire residual in $X$.
\end{maintheorem}

\smallskip

In the aftermath of the previous result, we secure a more precise description of the Lyapunov irregular points, that is, points that are non-typical for Oseledets' theorem. Given a compact metric space $(X,d)$, a continuous $f : X \to X$ and a continuous \emph{linear cocycle} $A: X \to GL(k, \mathbb R)$, $k \geqslant 2$, we assign to the cocycle the skew-product given by
$$\begin{array}{rccc}
F_A \colon & X \times \mathbb R^k & \to & M \times \mathbb R^k \\
	& (x,v) & \mapsto & (\, f(x), \, A(x)v \,).
\end{array}
$$
Set $A^n(x):= A(f^{(n-1)}(x)) \dots A(f(x))\, A(x)$ for each $x\in X$ and $n \in \mathbb{N}$. By Oseledets' theorem the \emph{largest Lyapunov exponent} (respectively the smallest Lyapunov exponent) associated to an $f$-invariant ergodic Borel probability measure $\mu$ is given by
$$\lambda^+(A,\mu)= \inf_{n \,\in \,\mathbb{N}}\, \frac1n \,\log \|A^n(x)\|	\quad \quad \text{ for } \mu \text{-a.e. $x$}$$
(respectively $\lambda^-(A,\mu)= \sup_{n \,\in \,\mathbb{N}} \,\frac1n\, \log \|A^n(x)^{-1}\|^{-1} \,\,\text{for } \mu \text{-a.e. $x$}$). Using Theorem~\ref{thm:mainB}, whenever the dynamics $f$ is minimal, we get additional information on the set of points whose Lyapunov exponent fails to be well defined.

\begin{maincorollary}\label{main-d}
Let $(X,d)$ be a compact metric space and $f: X \to X$ be a continuous minimal map. Given $k \geqslant 2$ and a continuous linear cocycle $A \in C^0(X,GL(k,\mathbb R))$, either
$$
\inf_{x\,\in\, X} \,\lim_{n\,\to\,+\infty}\, \frac1n \,\log \|A^n(x)\| = \sup_{x\,\in\, X}\, \lim_{n\,\to\,+\infty}\, \frac1n\, \log \|A^n(x)\|$$
and the previous value is the unique possible largest Lyapunov exponent associated to any $f$-invariant probability measure,
or there exists a Baire residual subset $\mathcal R \subset X$ such that
$$\forall \, x \in \mathcal R \quad \quad  \liminf_{n\,\to\,+\infty}\, \frac1n\, \log \|A^n(x)\| \,<\,\limsup_{n\,\to\,+\infty} \,\frac1n\, \log \|A^n(x)\|.$$
\end{maincorollary}

\smallskip

A dual statement, concerning smallest Lyapunov exponents, holds similarly. We note that Corollary~\ref{main-d} improves \cite[Theorem~4]{Fu97}, where the author considers $GL(2,\mathbb R)$-valued cocycles over minimal and uniquely ergodic maps, and whose argument is exclusive to $2$-dimensional linear cocycles. We also observe that, in a complementary direction, the non-existence of Lyapunov exponents for H\"older continuous matrix cocycles over maps satisfying exponential specification has been considered in \cite{Tian-c}. As far as we know, no other references address this problem of identification of topologically big sets of irregular points for sub-additive sequences of potentials.

\smallskip

As an illustration of the scope of Theorem~\ref{thm:mainB}'s applications, let us consider a continuous map $f: X\to X$ on a compact metric space, a sub-additive sequence of continuous functions $\Phi=(\varphi_n)_{n \, \in \, \mathbb{N}} \in C(X)^\N$ and a probability measure $\mu$ which is \emph{weak Gibbs} with respect to $\Phi$. For such a measure $\mu$ we may find a sequence of positive constants $(K_n)_{n \, \in \,\mathbb{N}}$ satisfying $\lim_{n\,\to\,+\infty} \,\frac{1}{n} \,\log K_n = 0$ and
\begin{equation}\label{eq:weak-Gibbs}
K_{n}^{-1}(x) \,\leqslant\, \frac{\mu(B_n(x,\varepsilon))}{e^{-n \,P(\Phi)\,+\,\varphi_n(x)}}\, \leqslant \, K_n(x) \quad \quad \forall \,x \in X \quad \forall \, n \, \in \,\mathbb{N}
\end{equation}
where $B_n(x,\varepsilon) := \big\{y \in X: \,d_n(x,y)\,<\,\varepsilon\big\}$ stands for the dynamical ball centered at $x$ of radius $\varepsilon$ and length $n$. These measures appear naturally in the context of equilibrium states for matrix cocycles $A: \,\{1, \dots, d\}^{\mathbb N} \to GL(d,\mathbb C)$ over the shift map $\sigma$ (cf. \cite{FK11}), with the weak-Gibbs condition then rewritten as
\begin{equation}\label{eq:weak-Gibbs}
K^{-1} e^{-n \,P(q)} \|A^{(n)}(x)\|^q \,\leqslant\, \frac{\mu(B_n(x,\varepsilon))}{e^{-n \,P(\Phi)\,+\,\varphi_n(x)}} \,\leqslant\, K e^{-n \,P(q)} \|A^{(n)}(x)\|^q
\end{equation}
where $A^{(n)}(x)=A(\sigma^{n-1}(x))\dots A(\sigma(x)) A(x)$
and $P(q)$ is the pressure function of the family $\Phi_q = (q\log \|A^n(\cdot)\|)_{n \,\in \,\mathbb{N}}$.

\smallskip

Now, when $\mu$ is $f$-invariant, the entropy $h_\mu(f)$ can be estimated using dynamical balls and the Brin-Katok formula \cite{BK1983} by
\begin{equation}\label{eq:BK}
h_\mu(f) = \lim_{\varepsilon\,\to \,0^+}\, \limsup_{n\,\to\, +\infty} \,-\frac1n \,\log \,\mu(B_n(x,\varepsilon))	\quad \quad \text{at $\mu\,$ a.e. } x
\end{equation}
and, if $f$ is expansive, the previous $\limsup$ does not depend on the value of $\vep$ if this is kept small enough. Moreover, in general, one has
$$\limsup_{n\,\to\,+\infty} \,-\frac1n \,\log\, \mu(B_n(x,\varepsilon)) = \limsup_{n\,\to\, +\infty} \,\frac1n\,\varphi_n(x) \quad \quad \forall \, x \in X$$
and
$$\liminf_{n\,\to\, +\infty} \,-\frac1n \,\log \,\mu(B_n(x,\varepsilon)) = \liminf_{n\,\to\,+\infty} \,\frac1n \,\varphi_n(x)\quad \quad \forall \, x \in X.$$
Therefore, if $f$ is, for instance, the one-sided full shift on a finite alphabet (so the pre-orbit of every point is dense) and $\mu$ is an $f$-invariant weak Gibbs measure, then Theorem~\ref{thm:mainB} implies that either the convergence to the metric entropy holds at all points, or the set of points for which the Brin-Katok limit \eqref{eq:BK} does not exist is a Baire residual subset of $X$.

\section{Preliminary result}\label{se:teo-p}

Let $(X,d)$ be a compact metric space and consider a sequence $(\psi_n)_{n \, \in \, \mathbb{N}}$ of continuous maps $\psi_n: X \to \mathbb{R}$. Define the map
$$x \in X \quad \mapsto \quad U_\psi(x) = \limsup_{n \, \in \, \mathbb{N}}\, \psi_n(x)$$
and consider the set
$$C_\psi = \Big\{x \in X  \colon\, \Big(\psi_n(x)\Big)_{n \, \in \, \mathbb{N}} \quad \mbox{does not converge}\Big\}.$$

\smallskip

\begin{theorem}\label{teo-p}
Assume that there exist two dense subsets $\mathcal{A}$ and $\mathcal{B}$ of $X$ such that the restrictions of the map $U_\psi$ to $\mathcal{A}$ and to $\mathcal{B}$ are constant, though the value at $\mathcal{A}$ is different from the one at $\mathcal{B}$. Then $C_\psi$ is a Baire residual subset of $X$.
\end{theorem}

\smallskip

\begin{proof}
Suppose that the constant value of $U_\psi$ at the dense sets $\mathcal{A}$ and $\mathcal{B}$ are $\alpha$ and $\beta$, respectively, with $\alpha \neq \beta$. Fix $0 <\vep < \frac16 \,\big|\alpha - \beta\big|$. Since the map $\psi_n$ is continuous for every $n \in \mathbb{N}$, given a positive integer $N$ the set
\begin{equation}\label{eq:ON}
\Lambda_N = \Big\{x \in X \colon \,\big|\psi_n(x) - \psi_m(x)\big| \leqslant \vep \quad \quad \forall \,m,n \geqslant N \Big\}
\end{equation}
is closed in $X$.

\smallskip

\begin{proposition}\label{prop:O-meagre}
$\Lambda_N$ has empty interior for every $N \in \mathbb{N}$.
\end{proposition}

\begin{proof}
Assume that there exists $N \in \mathbb{N}$ such that the interior of $\Lambda_N$, we denote by $\mathrm{int}(\Lambda_N)$, is non-empty, and take $\lambda \in \mathrm{int}(\Lambda_N)$. As $\mathcal{A}$ and $\mathcal{B}$ are dense in $X$, there exist  sequences $(p_n)_{n \, \in \, \mathbb{N}} \in \mathcal{A}$ and $(q_n)_{n \, \in \, \mathbb{N}} \in \mathcal{B}$ such that
\begin{equation}\label{eq:O-conv}
\forall\, n \in \mathbb{N} \quad p_n, \,q_n \,\in \,\mathrm{int}(\Lambda_N)  \quad \quad \text{and} \quad \quad \lim_{n\,\to\,+ \infty} \,p_n = \lambda = \lim_{n\,\to\,+ \infty} \,q_n.
\end{equation}

\smallskip

\begin{lemma}\label{O-le}
For every $N \in \mathbb{N}$, if a sequence $(x_k)_{k\,\in \,\mathbb{N}}$ of elements of $\Lambda_N$ converges, then
$$\bigg|\limsup_{k\,\to\,+ \infty}\, U_\psi(x_k) - U_\psi\Big(\lim_{k\,\to\,+ \infty} x_k\Big)\bigg| \leqslant 3\,\vep.$$
\end{lemma}

\smallskip

\begin{proof}[Proof of Lemma~\ref{O-le}]

Given $N \in \mathbb{N}$, take a convergent sequence $(x_k)_{k\,\in\, \mathbb{N}}$ contained in $\Lambda_N$ and
consider $\ell = \lim_{k\,\to\,\infty} \,x_k$, which is in $\Lambda_N$. By the definition of $\Lambda_N$, one has
\begin{eqnarray*}
\bigg|\psi_n(x_k) - \psi_m(x_k)\bigg| &\leqslant& \vep \quad \quad \forall\, m,n \geqslant N \quad \forall\, k \in \mathbb{N} \\
\bigskip
\bigg|\psi_n(\ell) - \psi_m(\ell)\bigg| &\leqslant& \vep \quad \quad \forall\, m,n \geqslant N.\\
\end{eqnarray*}
Fixing $m=N$ and taking the limit as $n$ goes to $+\infty$ in the first inequality along subsequences that attain the $\limsup$, we conclude that
$$\forall \, k \in \mathbb{N} \quad \bigg|U_\psi(x_k) - \psi_N(x_k))\bigg| \leqslant \vep
\quad \quad \text{and} \quad \quad 	\bigg|U_\psi(\ell) - \psi_N(\ell)\bigg| \leqslant \vep.$$
By the compactness of $(X,d)$ and the uniform continuity of $\psi_N$, we may choose $\delta_N > 0$ such that
$$d(z,w)<\delta_N \quad \Rightarrow \quad \Big|\psi_N(z) -\psi_N(w)\Big| < \vep.$$
Altogether, this proves that, for $k \in \mathbb{N}$ large enough so that $d(x_k, \ell) < \delta_N$, one has
$$\Big|U_\psi(x_k) - U_\psi(\ell)\Big| \leqslant 3\,\vep.$$
In particular,
$$\bigg|\limsup_{k\,\to\,+ \infty}\, U_\psi(x_k) - U_\psi(\ell)\bigg| \leqslant 3\,\vep$$
as claimed.
\end{proof}

\bigskip

Let us resume the proof of the Proposition~\ref{prop:O-meagre}. As $U_\psi(p_n)=\alpha$ and $U_\psi(q_n)=\beta$ for every $n \in \mathbb{N}$, the conditions \eqref{eq:O-conv} and Lemma~\ref{O-le} imply that
$$|\alpha - U_\psi(\lambda)| \leqslant 3\,\vep  \quad \text{ and } \quad |\beta - U_\psi(\lambda)| \leqslant 3\,\vep.$$
So $|\alpha - \beta| \leqslant 6\,\vep$, contradicting the choice of $\vep$. Thus, $\Lambda_N$ must have empty interior. This completes the proof of the proposition.
\end{proof}

Finally, observe that the set of points $x \in X$ whose sequence $\big(\psi_n(x)\big)_{n \, \in \, \mathbb{N}}$ converges is contained in the countable union $\bigcup_{N=1}^{+\infty} \,\Lambda_N$ of closed sets with empty interior. 
This ends the proof of Theorem~\ref{teo-p}.
\end{proof}

\section{Proof of Theorem~\ref{thm:main}}\label{proof-main}

Let $(X,d)$ be a compact metric space and $f: X \to X$ be a continuous map and $\varphi: X \to \mathbb{R}$ be a continuous observable such that there exist two dense subsets $\mathcal{A}$ and $\mathcal{B}$ of $X$ such that the restrictions of the map $L_\varphi$ to $\mathcal{A}$ and to $\mathcal{B}$ are constant, equal to $\alpha$ and $\beta$, respectively, and $\alpha \neq \beta$. To prove Theorem~\ref{thm:main} we just run the argument used to show Theorem~\ref{teo-p} with the following adaptations:
\begin{enumerate}
\item The sequence $(\psi_n)_{n \, \in \, \mathbb{N}}$ is made of the Birkhoff averages of $f$ with respect to $\varphi$, that is, for every $x \in X$ and every $n \in \mathbb{N}$,
    $$\psi_n(x) = \frac1n\, \sum_{j=0}^{n-1} \,\varphi(f^j (x)).$$
\item The map $U_\psi$ is precisely $L_\varphi$ (cf. definition \eqref{eq:function L}).
\medskip
\item For every $N \in \mathbb{N}$, the set $\Lambda_N$ is now defined by
$$\Lambda_N = \Big\{x \in X \colon \bigg| \frac1n \, \sum_{j=0}^{n-1} \,\varphi(f^j (x)) - \frac1m \,\sum_{j=0}^{m-1}\, \varphi(f^j (x))\bigg| \leqslant \vep \quad \quad \forall \,m,n \geqslant N \Big\}.$$
\item The set $C_\psi$ becomes $\mathfrak I_\varphi$ (cf. definition \eqref{eq:IF}).
\end{enumerate}

\smallskip

\section{Proof of Corollary~\ref{main-b}}\label{proof2}

Let $(X,d)$ be a compact metric space, $f: X \to X$ be a continuous map and $\mu_1$ and $\mu_2$ two distinct $f$-invariant Borel probability measures whose ergodic basins $\mathbb{B}(\mu_1)$ and $\mathbb{B}(\mu_2)$ are dense in $X$. Choose $\varphi \in C^0(X,\mathbb R)$ such that $\int \varphi\, d\mu_1 \neq \int \varphi\, d\mu_2$, and consider $\mathcal{A} = \mathbb{B}(\mu_1)$ and $\mathcal{B} = \mathbb{B}(\mu_2).$
In $\mathcal{A}$ and $\mathcal{B}$ the map $L_\varphi$ is constant, and equal to $\int \varphi\, d\mu_1$ and $\int \varphi\, d\mu_2$, respectively. Moreover, by assumption, these sets are dense in $X$. Thus, Theorem~\ref{thm:main} guarantees that $\mathfrak I_\varphi$ is Baire residual in $X$, and so $\mathfrak I$ is generic as well.

\begin{remark}
\emph{Another proof of this corollary could be obtained as follows. For each $f$-invariant Borel probability measure $\mu$ take $\mathcal E_N(\mu)$ as the set of points $x \in X$ for which there exists $n > N$ such that $\big| \frac1n \sum_{j=0}^{n-1} \varphi(f^j(x)) - \int \varphi\, d\mu\big| <\frac1N$. If there are $\mu_1 \neq \mu_2$ with dense basins, then the set $\mathcal E_N(\mu_1)\cap \mathcal E_N(\mu_2)$ is open and dense in $X$ for every $N \geqslant 1$. Moreover, the set of $\varphi$-irregular points contains the Baire residual subset $\bigcap_{N \, \geqslant \, 1}\, \mathcal E_N(\mu_1)\cap \mathcal E_N(\mu_2)$.}
\end{remark}

\section{Proof of Corollary~\ref{main-a}}\label{proof1}

Let $(X,d)$ be a compact metric space and $f: X \to X$ be a continuous minimal map. If $f$ is uniquely ergodic and $\mu$ denotes the unique $f$-invariant probability measure, then the sequences of Birkhoff averages of every continuous observable map $\varphi \in C^0(X,\mathbb R)$ are uniformly convergent in $X$ to the constant $\int_X \varphi \, d\mu$. Thus, in the weak$^*$ topology, one has
$$\lim_{n \,\to \,+\infty} \,\frac1n\,\sum_{j=0}^{n-1} \,\delta_{f^j(x)}  = \mu \quad \quad \forall \,x\in X.$$

\smallskip

Assume now that there exist two distinct $f$-invariant Borel ergodic probability measures $\mu_1$ and $\mu_2$. Choose $\varphi \in C^0(X,\mathbb R)$ such that $\int \varphi\, d\mu_1 \neq \int \varphi\, d\mu_2$ and take two points $p_1 \in \mathbb{B}(\mu_1)$ and $p_2 \in \mathbb{B}(\mu_2)$, where $\mathbb{B}(\mu_i)$ denotes the ergodic basin of attraction of $\mu_i$. These are dense subsets of $X$, since each contains a dense orbit, within which $L_\varphi$ is constant, equal to $\int \varphi\, d\mu_1$ in $\mathbb{B}(\mu_1)$ and to $\int \varphi\, d\mu_2$ in $\mathbb{B}(\mu_2)$. Therefore, we may apply Corollary~\ref{main-b}, concluding that the set $\mathfrak I$ is Baire residual in $X$, as claimed.

\smallskip

\begin{remark}\label{rmk51}
\emph{One may ask whether results analogous to Corollary~\ref{main-a} hold for non-additive sequences (e.g. almost additive, asymptotically additive or sub-additive). In view of the very recent work \cite{Noe}, the limits of almost additive or asymptotically additive sequence of continuous maps coincide with Birkhoff averages of a suitable continuous observable. In particular, the conclusion of Corollary~\ref{main-a} is valid for this more general class of limits and sequences. However, if the sequences are just sub-additive, then this is no longer true. Indeed, for any uniquely ergodic system there exists sub-additive sequences of continuous maps such that the set of non-typical points is Baire generic (cf. \cite{Fu97}).}
\end{remark}

\section{Proof of Corollary~\ref{main-c}}\label{proof3}

Let $(X,d)$ be a compact Riemannian manifold and $f: X \to X$ be a partially hyperbolic diffeomorphism. Suppose that $f$ has two distinct periodic points $p_1$ and $p_2$ whose stable manifolds, we denote by $W_f^s(p_1)$ and $W_f^s(p_2)$, respectively, are dense in $X$. We may assume that $p_1$ and $p_2$ are fixed by $f$, taking an appropriate power of $f$ otherwise. Choose $\varphi \in C^0(X,\mathbb R)$ such that $\varphi(p_1) \neq \varphi(p_2)$ and define $\mathcal{A} = W_f^s(p_1)$ and $\mathcal{B} = W_f^s(p_2).$ These are $f$-invariant, dense subsets of $X$, and the map $L_\varphi$ is constant in each of them, equal to $\varphi(p_1)$ and $\varphi(p_2)$, respectively. Indeed, as $p_1$ and $p_2$ are fixed points by $f$ and $\varphi$ is continuous, then one has $L_\varphi(x)=L_\varphi(p_1)=\varphi(p_1)$ for every $x \in W_f^s(p_1)$ (and, analogously, $L_\varphi(y)=L_\varphi(p_2)=\varphi(p_2)$ for every $y \in W_f^s(p_2)$) due to the following immediate chain of deductions:
\begin{eqnarray*}
x \in W_f^s(p_1) \quad &\Leftrightarrow& \quad \lim_{n\,\to\,+\infty}\, f^n(x)=p_1 \quad \Rightarrow \quad \lim_{n\,\to\,+\infty}\, \varphi(f^n(x))=\varphi(p_1) \\
&\Rightarrow& \quad \lim_{n \,\to \,+\infty} \,\frac1n \,\sum_{j=0}^{n-1}\, \varphi({f^j(x)})=\varphi(p_1).
\end{eqnarray*}
Consequently, Theorem~\ref{thm:main} ensures that $\mathfrak I$ is a Baire residual subset of $X$.

\section{Proof of Theorem~\ref{thm:mainB}}\label{proofB}

Let $(X,d)$ be a compact metric space and $f: X \to X$ be a continuous map and $\Phi:= (\varphi_n)_{n \,\in \, \mathbb{N}}$ be a sub-additive sequence of continuous functions.
Assume that $\mathcal{A}$ and $\mathcal{B}$ are dense subsets of $X$ such that the restrictions of $M_\Phi$ to $\mathcal{A}$ and to $\mathcal{B}$ are constant, equal to $\alpha$ and $\beta$, respectively, and $\alpha \neq \beta$. Again, to show Theorem~\ref{thm:mainB} we repeat the proof of Theorem~\ref{teo-p} after some adjustments:
\begin{enumerate}
\item The sequence $(\psi_n)_{n \, \in \, \mathbb{N}}$ is now $(\frac1n\,\varphi_n)_{n \,\in \, \mathbb{N}}$.
\smallskip
\item As the sequence $(\psi_n)_{n \, \in \, \mathbb{N}}$ is sub-additive, the map $U_\psi$ coincides with $M_\varphi$ (cf. definition \eqref{eq:function LP}) whenever the limit of the sequence exists.
\medskip
\item For every $N \in \mathbb{N}$, the set $\Lambda_N$ is given by
$$\Lambda_N = \Big\{x \in X \colon \bigg| \frac1n \, \varphi_n(x) - \frac1m \,\varphi_m(x)\bigg| \leqslant \vep \quad \quad \forall \,m,n \geqslant N \Big\}.$$
\item The set $C_\psi$ is precisely what we denoted by $\mathfrak I_\Phi$ (cf. definition \eqref{eq:IFP}).
\end{enumerate}

\section{Proof of Corollary~\ref{main-homoclinic}}\label{proof-homoclinic}

Let $X$ be a compact Riemannian manifold and $f\in \text{Diff}^{\,1}(X)$. Assume that $H(p,f)\neq\{p\}$ is
a homoclinic class for $f$ associated to a hyperbolic saddle periodic point $p$. By Birkhoff's theorem, every transversal homoclinic point in $H(p,f)$ is accumulated by hyperbolic periodic orbits with the same index (that is, the dimension of the stable subbundle) as $p$. In particular, there exists a hyperbolic saddle $q \in H(p,f)$ which does not belong to the orbit $\mathcal O(p)$ of $p$ and is homoclinically related to $p$. This ensures that $W^s(\mathcal O(p)) \pitchfork W^u(\mathcal O(q)) \neq \emptyset$ and $W^s(\mathcal O(q)) \pitchfork W^u(\mathcal O(p)) \neq \emptyset$. Now, the $\lambda$-lemma guarantees that $W^s(\mathcal O(q))$ is dense in $H(p,f)$. This property together with Theorem~\ref{thm:main}
ensure that $\mathfrak{I}_\varphi \cap H(p,f)$ is a Baire residual subset of $H(p,f)$ for every continuous observable $\varphi$ whose averages along the orbits of the periodic points $p$ and $q$ differ. This proves the corollary.

\section{Proof of Corollary~\ref{main-d}}\label{proofd}

Let $(X,d)$ be a compact metric space, $f: X \to X$ be a continuous minimal map and $A \in C^0(X,GL(k, \mathbb R))$, for $k \geqslant 2$. Consider the sub-additive sequence $\Phi_A=(\varphi_n)_{n \,\in \,\mathbb N}$, where
$$\varphi_n(x)= \log \|A^n(x)\| \quad \quad \forall \, x \in X.$$
On the one hand, by the sub-additivity of the sequence $\Big(\log \|A^n(x)\|\Big)_{n \, \in \, \mathbb{N}}$ we conclude that
the map
$$M_{\Phi_A}(\cdot) = \inf_{n \, \in \, \mathbb{N}}\, \frac1n \,\log \|A^n(\cdot)\|$$
satisfies $M_{\Phi_A}(x) \leqslant M_{\Phi_A}(f(x))$ for every $x\in X$. On the other hand, for each $x\in X$, the reverse inequality $M_{\Phi_A}(x) \geqslant M_{\Phi_A}(f(x))$ follows from the estimate
$$\|A^{n}(x) \| = \|A^{n-1}(f(x)) A(x)\| \,\geqslant \, \min_{z\,\in\, X} \,\|A(z)^{-1}\|^{-1} \,\cdot\, \|A^{n-1}(f(x))\| \quad \quad \forall \, n \in \mathbb N.$$
Therefore, $M_{\Phi_A}$ is a first integral with respect to $f$. In particular, by the minimality of $f$, if there are $x_1 \neq x_2 \in X$ such that
$$\lim_{n\, \to\, +\infty} \,\frac1n \,\log \|A^n(x_1)\| \,<\, \lim_{n\, \to\, +\infty}\, \frac1n\, \log \|A^n(x_2)\|$$
then there exist dense sets $\mathcal A$ and $\mathcal B$ (namely, the orbits of $x_1$ and $x_2$ by $f$) satisfying the requirements of Theorem~\ref{thm:mainB}. Thus $\mathfrak I_\Phi$ is Baire residual in $X$.

\section{The case of flows}\label{se:flows}

The proofs of Theorems~\ref{thm:main} and ~\ref{thm:mainB} apply \emph{verbatim} to continuous $\mathbb R$-actions if one replaces Birkhoff averages by the suitable means obtained by integration along the orbits of the flow. Nevertheless, in some particular cases one can reduce the analysis of the continuous-time to the discrete-time setting. In order to illustrate such an application, given a continuous flow $(Y_t)_{t\,\in\, \mathbb R}$ on a compact metric space $X$ define the irregular set by
\begin{equation}
\mathfrak I = \Big\{x \in X  \colon \,\lim_{t \,\to \,+\infty} \,\frac1t \int_0^t \,\delta_{Y_s(x)} \, ds \quad \mbox{does not exist in the weak$^*$ topology}\, \Big\}.
\end{equation}
If there exists a pair of distinct Borel probability measures with dense ergodic basins, the situation can be reduced to the discrete-time setting in the sense that the following result is a consequence of Corollary~\ref{main-b}.

\begin{maincorollary}\label{main-f}
Let $(X,d)$ be a compact metric space and $(Y_t)_{t\,\in\, \mathbb R}$ be a continuous flow on $X$ preserving two distinct Borel probability measures with dense ergodic basins. Then $\mathfrak I$ is a Baire residual subset of $X$.
\end{maincorollary}

\begin{proof}
Let $\mu_1$ and $\mu_2$ be two distinct $(Y_t)_{t\,\in\, \mathbb R}$-invariant Borel probability measures whose ergodic basins $\mathbb{B}(\mu_1)$ and $\mathbb{B}(\mu_2)$ are dense in $X$. Pick $\varphi \in C^0(X,\mathbb R)$ such that $\int \varphi\, d\mu_1 \neq \int \varphi\, d\mu_2$. As a consequence of \cite{PS}, for each $i=1,2$ there exists a Baire residual subset $\mathcal R_i \subset \mathbb{R}$ of times such that $\mu_i$ is ergodic for the time-$t$ map $Y_t$ associated to $t \in \mathcal R_i$. Now, fix an arbitrary $T \in \mathcal R_1\cap \mathcal R_2$, and consider the homeomorphism $f = Y_T$ and the potential $\varphi_T:= \frac1T \int_0^T \varphi(Y_s(x))\, ds$.

\smallskip

The sets $\mathcal{A} = \mathbb{B}(\mu_1)$ and $\mathcal{B} = \mathbb{B}(\mu_2)$ are $f$-invariant and, by assumption, dense in $X$. Besides, by the ergodicity of $\mu_1$ and $\mu_2$ with respect to $f$, the map $L_{\varphi_T}$ is constant in $\mathcal{A}$ and $\mathcal{B}$, and equal to $\int \varphi\, d\mu_1$ and $\int \varphi\, d\mu_2$, respectively. Thus, Corollary~\ref{main-b} implies that $\mathfrak I$ is Baire residual in $X$.
\end{proof}

Corollary~\ref{main-f} implies, along the same lines used in the proof of Corollary~\ref{main-homoclinic}, that every
non-trivial homoclinic class of a vector field has a Baire generic subset of points with historic behavior.
Since every singular-hyperbolic attractor in dimension three is a homoclinic class (cf. \cite{AP} for the definition and proofs),
we conclude the following:

\begin{maincorollary}\label{main-g}
Let $M$ be a three-dimensional compact Riemannian manifold, $(X_t)_t$ be a $C^1$-smooth flow and
$\Lambda$ be a singular-hyperbolic attractor. The set of points with historic behavior is Baire residual in $\Lambda$.
\end{maincorollary}

It is well known that geometric Lorenz attractors are homoclinic classes. In particular, this improves \cite{KLS} where the authors prove that the set of points with historic behavior for the geometric Lorenz attractor is residual in a trapping region of the attractor.
We also observe that a similar statement holds for multidimensional singular-hyperbolic attractors of $C^2$ vector fields,
extending the recent result of D. Yang in \cite{Y2020}. Actually, by \cite[Theorem B]{CY}, for a $C^1$ open and dense set of vector fields their singular-hyperbolic Lyapunov stable chain-recurrence class is a homoclinic class.

%
%
%

\subsection*{Acknowledgements}
MC and PV were partially supported by CMUP which is financed by national funds through FCT - Funda\c c\~ao para a Ci\^encia e a Tecnologia, I.P., under the project with reference UIDB/00144/2020. PV also acknowledges financial support from the Project `New trends in Lyapunov exponents' (PTDC/MAT-PUR/29126/2017) and from Funda\c c\~ao para a Ci\^encia e Tecnologia (FCT) - Portugal through the grant CEECIND/03721/2017 of the Stimulus of
Scientific Employment, Individual Support 2017 Call.

\end{document}